\documentclass[11pt]{amsart}
\usepackage{amsopn}
\usepackage{amssymb, amscd}
\usepackage{amsmath}
\usepackage{graphicx, graphics, epsfig}

\newcommand{\nc}{\newcommand}

\nc{\eg}{\mathfrak{e} } \nc{\fg}{\mathfrak{f} } \nc{\vg}{\mathfrak{v} } \nc{\wg}{\mathfrak{w} }
\nc{\zg}{\mathfrak{z} } \nc{\ngo}{\mathfrak{n} } \nc{\kg}{\mathfrak{k} }
\nc{\mg}{\mathfrak{m} } \nc{\bg}{\mathfrak{b} } \nc{\ggo}{\mathfrak{g} }
\nc{\ggob}{\overline{\mathfrak{g}} } \nc{\sog}{\mathfrak{so} }
\nc{\sug}{\mathfrak{su} } \nc{\spg}{\mathfrak{sp} } \nc{\slg}{\mathfrak{sl} }
\nc{\glg}{\mathfrak{gl} } \nc{\cg}{\mathfrak{c} } \nc{\rg}{\mathfrak{r} }
\nc{\hg}{\mathfrak{h} } \nc{\tg}{\mathfrak{t} } \nc{\ug}{\mathfrak{u} }
\nc{\dg}{\mathfrak{d} } \nc{\ag}{\mathfrak{a} } \nc{\pg}{\mathfrak{p} }
\nc{\sg}{\mathfrak{s} } \nc{\affg}{\mathfrak{aff} }

\nc{\pca}{\mathcal{P}} \nc{\nca}{\mathcal{N}} \nc{\lca}{\mathcal{L}}
\nc{\oca}{\mathcal{O}} \nc{\mca}{\mathcal{M}} \nc{\tca}{\mathcal{T}}
\nc{\aca}{\mathcal{A}} \nc{\cca}{\mathcal{C}} \nc{\gca}{\mathcal{G}}
\nc{\sca}{\mathcal{S}} \nc{\hca}{\mathcal{H}} \nc{\bca}{\mathcal{B}}
\nc{\dca}{\mathcal{D}} \nc{\val}{\operatorname{val}}

\nc{\vp}{\varphi} \nc{\ddt}{\frac{d}{dt}} \nc{\dds}{\frac{d}{ds}}
\nc{\dpar}{\frac{\partial}{\partial t}} \nc{\im}{\mathtt{i}}

\nc{\SO}{\mathrm{SO}} \nc{\Spe}{\mathrm{Sp}} \nc{\Sl}{\mathrm{SL}}
\nc{\SU}{\mathrm{SU}} \nc{\Or}{\mathrm{O}} \nc{\U}{\mathrm{U}} \nc{\Gl}{\mathrm{GL}}
\nc{\Se}{\mathrm{S}} \nc{\Cl}{\mathrm{Cl}} \nc{\Spein}{\mathrm{Spin}}
\nc{\Pin}{\mathrm{Pin}} \nc{\G}{\mathrm{GL}_n(\RR)} \nc{\g}{\mathfrak{gl}_n(\RR)}

\nc{\RR}{{\Bbb R}} \nc{\HH}{{\Bbb H}} \nc{\CC}{{\Bbb C}} \nc{\ZZ}{{\Bbb Z}}
\nc{\FF}{{\Bbb F}} \nc{\NN}{{\Bbb N}} \nc{\QQ}{{\Bbb Q}} \nc{\PP}{{\Bbb P}}

\nc{\vs}{\vspace{.2cm}} \nc{\vsp}{\vspace{1cm}} \nc{\ip}{\langle\cdot,\cdot\rangle}
\nc{\ipp}{(\cdot,\cdot)} \nc{\la}{\langle} \nc{\ra}{\rangle} \nc{\unm}{\tfrac{1}{2}}
\nc{\unc}{\tfrac{1}{4}} \nc{\und}{\tfrac{1}{16}} \nc{\no}{\vs\noindent}
\nc{\lamkn}{\Lambda^2(\RR^{q+n})^*\otimes\RR^{q+n}} \nc{\lamn}{\Lambda^2(\RR^n)^*\otimes\RR^n} \nc{\lamp}{\Lambda^2\pg^*\otimes\pg}
\nc{\lamg}{\Lambda^2\ggo^*\otimes\ggo} \nc{\lamngo}{\Lambda^2\ngo^*\otimes\ngo}
\nc{\tangz}{{\rm T}^{\rm Zar}} \nc{\mum}{/\!\!/} \nc{\kir}{/\!\!/\!\!/}
\nc{\Ri}{\tfrac{4\Ric_{\mu}}{||\mu||^2}} \nc{\ds}{\displaystyle}
\nc{\ben}{\begin{enumerate}} \nc{\een}{\end{enumerate}} \nc{\f}{\frac}
\nc{\lb}{[\cdot,\cdot]} \nc{\isn}{\tfrac{1}{||v||^2}}
\nc{\gkp}{(\ggo=\kg\oplus\pg,\ip)} \nc{\ukh}{(\ug=\kg\oplus\hg,\ip)}
\nc{\raw}{\rightarrow} \nc{\lraw}{\longrightarrow}

\nc{\Hess}{\operatorname{Hess}} \nc{\ad}{\operatorname{ad}}
\nc{\Ad}{\operatorname{Ad}} \nc{\rank}{\operatorname{rank}}
\nc{\Irr}{\operatorname{Irr}} \nc{\End}{\operatorname{End}}
\nc{\Aut}{\operatorname{Aut}} \nc{\Inn}{\operatorname{Inn}}
\nc{\Der}{\operatorname{Der}} \nc{\Ker}{\operatorname{Ker}}
\nc{\Iso}{\operatorname{I}} \nc{\Diff}{\operatorname{Diff}}
\nc{\Lie}{\operatorname{Lie}} \nc{\tr}{\operatorname{tr}} \nc{\dif}{\operatorname{d}}
\nc{\sen}{\operatorname{sen}} \nc{\modu}{\operatorname{mod}}
\nc{\Riem}{\operatorname{Rm}} \nc{\Ricci}{\operatorname{Ric}}
\nc{\sym}{\operatorname{sym}} \nc{\symac}{\operatorname{sym^{ac}}}
\nc{\symc}{\operatorname{sym^{c}}} \nc{\scalar}{\operatorname{R}}
\nc{\grad}{\operatorname{grad}} \nc{\ricci}{\operatorname{Rc}}
\nc{\nr}{\operatorname{nr}} \nc{\riccic}{\operatorname{ric^{c}}}
\nc{\riccig}{\operatorname{ric^{\gamma}}} \nc{\Rin}{\operatorname{M}}
\nc{\Le}{\operatorname{L}}
\nc{\level}{\operatorname{level}} \nc{\rad}{\operatorname{rad}}
\nc{\abel}{\operatorname{ab}} \nc{\CH}{\operatorname{CH}}
\nc{\mcc}{\operatorname{mcc}} \nc{\Adj}{\operatorname{Adj}}
\nc{\Order}{\operatorname{O}} \nc{\mm}{\operatorname{M}}
\nc{\inj}{\operatorname{inj}} \nc{\proy}{\operatorname{pr}}
\nc{\vol}{\operatorname{vol}}\nc{\diag}{\operatorname{Diag}}

\theoremstyle{plain}
\newtheorem{theorem}{Theorem}[section]
\newtheorem{proposition}[theorem]{Proposition}
\newtheorem{corollary}[theorem]{Corollary}
\newtheorem{lemma}[theorem]{Lemma}

\theoremstyle{definition}
\newtheorem{definition}[theorem]{Definition}

\theoremstyle{remark}
\newtheorem{remark}[theorem]{Remark}

\newtheorem{example}[theorem]{Example}

\title{On homogeneous Ricci solitons}

\author{Ramiro Lafuente} \author{Jorge Lauret}

\address{Universidad Nacional de C\'ordoba, FaMAF and CIEM, 5000 C\'ordoba, Argentina}
\email{rlafuente@famaf.unc.edu.ar, lauret@famaf.unc.edu.ar}

\thanks{This research was partially supported by grants from CONICET, FONCYT and SeCyT (Universidad Nacional de C\'ordoba)}

\begin{document}

\maketitle

\begin{abstract}
We study the evolution of homogeneous Ricci solitons under the bracket flow, a dynamical system on the space $\hca_{q,n}\subset\lamg$ of all homogeneous spaces of dimension $n$ with a $q$-dimensional isotropy, which is equivalent to the Ricci flow for homogeneous manifolds.  We prove that algebraic solitons (i.e. the Ricci operator is a multiple of the identity plus a derivation) are precisely the fixed points of the system, and that a homogeneous Ricci soliton is isometric to an algebraic soliton if and only if the corresponding bracket flow solution is not chaotic, in the sense that its $\omega$-limit set consists of a single point.  We also geometrically characterize algebraic solitons among homogeneous Ricci solitons as those for which the Ricci flow solution is simultaneously diagonalizable.
\end{abstract}

\tableofcontents

\section{Introduction}

Ricci solitons are precisely those Riemannian metrics that are `nice' enough to be `upgraded' by the Ricci flow, in the sense that they just evolve by scaling and pull-back by diffeomorphisms.  A good understanding of them is therefore crucial to study the singularity behavior of the Ricci flow on any class of manifolds.  Trivial examples are provided by Einstein metrics, as well as by direct products of an Einstein manifold with any Euclidean space.  Non-K\"ahler examples of Ricci solitons are still very hard to find (see \cite{DncHllWng}).

Homogeneity seems to be a very strong condition to impose.  Nontrivial homogeneous Ricci solitons must be expanding (i.e. the scaling is time increasing) and they can never be compact nor gradient.  One could say that they should not exist, but they do.  Actually, the nilpotent part of any Einstein solvable Lie group gives an example of a homogeneous Ricci soliton, known as a {\it nilsoliton} in the literature (see \cite{soliton}).  One may also extend a nilsoliton to a different solvable Lie group and obtain Ricci solitons which are not Einstein, the so called {\it solvsolitons} (see \cite{solvsolitons}).  In all these examples the following `algebro-geometric' condition holds:
\begin{equation}\label{ricder}
\Ricci(g)=cI+D, \qquad\mbox{for some}\quad c\in\RR,\quad D\in\Der(\sg),
\end{equation}
once the left-invariant metric $g$ is identified with an inner product on the Lie algebra $\sg$ of the solvable Lie group $S$.  So far, simply connected solvsolitons are the only known examples of nontrivial homogeneous Ricci solitons.

A homogeneous space $(G/K,g)$ is said to be a {\it semi-algebraic soliton} if there exists a one-parameter family of equivariant diffeomorphisms $\vp_t\in\Aut(G/K)$ (i.e. automorphisms of $G$ taking $K$ onto $K$) such that $g(t)=c(t)\vp_t^*g$ is a solution to the Ricci flow starting at $g(0)=g$ for some scaling function $c(t)>0$.  It is called an {\it algebraic soliton} if in addition, for some reductive decomposition $\ggo=\kg\oplus\pg$, the derivatives $d\vp_t|_o:\pg\longrightarrow\pg$ are all symmetric.  The notion of algebraic soliton is precisely the generalization of condition (\ref{ricder}) to any Lie group or homogeneous space.  It has recently been proved in \cite{Jbl} that any homogeneous Ricci soliton $(M,g)$ is semi-algebraic with respect to its full isometry group $G=\Iso(M,g)$.

We give in Section \ref{hrs} an up-to-date overview on homogeneous Ricci solitons.  Next, in Section \ref{ASBF}, we study the evolution of semi-algebraic solitons under the {\it bracket flow}, an ODE for a family $\mu(t)\in\hca_{q,n}\subset\lamg$.  Here $\hca_{q,n}$ denotes the subset of the variety of Lie brackets on the fixed vector space $\ggo$ parameterizing the space of all homogeneous spaces of dimension $n$ with a $q$-dimensional isotropy (see \cite{spacehm}).  This dynamical system has been proved in \cite{homRF} to be equivalent in a precise sense to the Ricci flow (preliminaries on this machinery are given in Section \ref{hm}).  We first show that algebraic solitons are precisely the fixed points, and hence the possible limits of any normalized bracket flow.  This in particular yields their asymptotic behavior and implies that algebraic solitons are all {\it Ricci flow diagonal}, in the sense that the Ricci flow solution $g(t)$ simultaneously diagonalizes with respect to a fixed orthonormal basis of some tangent space.

Furthermore, given a starting point $\mu_0\in\hca_{q,n}$, we prove that one can obtain at most one nonflat algebraic soliton $\lambda$ as a limit by running all possible normalized bracket flow solutions $\mu(t)$.  By using \cite{spacehm}, we can translate this convergence  $\mu(t)\to\lambda$ of Lie brackets into more geometric notions of convergence, including convergence on the pointed or Cheeger-Gromov topology.  The limit Lie bracket $\lambda$ might be non-isomorphic to $\mu(t)$ and therefore provides an explicit limit $(G_\lambda/K_\lambda,g_\lambda)$ which is often non-diffeomorphic and even non-homeomorphic to the starting homogeneous manifold $(G_{\mu_0}/K_{\mu_0},g_{\mu_0})$ .

Regarding semi-algebraic solitons, we obtain that under any normalized bracket flow, they simply evolve by
$$
\mu(t)=e^{tA}\cdot\mu_0=e^{tA}\mu_0(e^{-tA}\cdot, e^{-tA}\cdot),
$$
for some skew-symmetric map $A:\ggo\longrightarrow\ggo$.  Furthermore, they are algebraic solitons if and only if $A\in\Der(\mu_0)$ (i.e. fixed points).  In particular, any homogeneous Ricci soliton would be necessarily isometric to an algebraic soliton in case the bracket flow was not chaotic, in the sense that the $\omega$-limit set of any solution is a single point.  This is an open question.

Whereas being a Ricci soliton is invariant under isometry, the concept of semi-algebraic soliton is not, as it may depend on the presentation of the homogeneous manifold $(M,g)$ as a homogeneous space $(G/K,g)$.  We prove in Section \ref{algdiag} that the property of being Ricci flow diagonal characterizes algebraic solitons.  Namely:

\begin{quote}
A homogeneous Ricci soliton is Ricci flow diagonal if and only if it is isometric to an algebraic soliton.
\end{quote}

Note that this is a geometric characterization as the property of being Ricci flow diagonal is also invariant under isometry.  The Ricci flow evolution of a homogeneous Ricci soliton which is not isometric to any algebraic soliton, if any, is therefore quite different from all known examples (i.e. solvsolitons).

\vs \noindent {\it Acknowledgements.} We are very grateful to M. Jablonski for fruitful discussions on the topic of this paper.

\section{Preliminaries}\label{hm}

Our aim in this section is to briefly describe a framework developed in \cite{spacehm} which allows us to work on the `space of homogeneous manifolds', by parameterizing  the set of all simply connected homogeneous spaces of dimension $n$ and isotropy dimension $q$ by a subset $\hca_{q,n}$ of the variety of $(q+n)$-dimensional Lie algebras.  According to the results in \cite{homRF}, the Ricci flow is equivalent to an ODE system on $\hca_{q,n}$ called the bracket flow.

Given a connected homogeneous Riemannian manifold $(M,g)$, each transitive closed Lie subgroup $G\subset\Iso(M,g)$ gives rise to a
presentation of $(M,g)$ as a homogeneous space $(G/K,g)$, where
$K$ is the isotropy subgroup of $G$ at some point $o\in M$ and $g$ becomes a $G$-invariant metric.

As $K$ is compact, there always
exists a {\it reductive} (i.e. $\Ad(K)$-invariant) decomposition $\ggo=\kg\oplus\pg$, where $\ggo$ and $\kg$ are respectively the Lie algebras of $G$ and $K$.  Thus $\pg$ can be naturally identified with the tangent space $\pg\equiv T_oM=T_oG/K$, by taking the value at the origin $o=eK$ of the Killing vector fields
corresponding to elements of $\pg$ (i.e. $X_o=\ddt|_0\exp{tX}(o)$).  Let $g_{\ip}$ denote the $G$-invariant metric on $G/K$ determined by $\ip:=g(o)$, the $\Ad(K)$-invariant inner product on $\pg$ defined by $g$.  In this situation, when a reductive decomposition has already been chosen, the homogeneous space will be denoted by $(G/K,g_{\ip})$.

In order to get a presentation $(M,g)=(G/K,g_{\ip})$ of a connected homogeneous manifold as a homogeneous space, there is no need for $G\subset\Iso(M,g)$ to hold, that is, an {\it effective} action.  It is actually enough to have a transitive action of $G$ on $M$ by isometries being {\it almost-effective} (i.e. the normal subgroup $\{ g\in G:ghK=hK, \;\forall h\in G\}$ of $K$ is discrete), along with a reductive decomposition $\ggo=\kg\oplus\pg$ such that the inner product $\ip$ on $\pg$ defined by $\ip:=g(o)$ is $\Ad(K)$-invariant.  Any homogeneous space considered in this paper will be assumed to be almost-effective and connected.

In the study of homogeneous Ricci solitons carried out in the present paper, the following special reductive decomposition has often been very convenient to take.

\begin{lemma}\label{Bkp0}
Let $(G/K,g)$ be a homogeneous space.  Then there exists a reductive decomposition  $\ggo=\kg\oplus\pg$ such that $B(\kg,\pg)=0$, where $B$ is the Killing form of $\ggo$.
\end{lemma}

\begin{proof}
It follows by taking $\pg$ as the orthogonal complement of $\kg$ in $\ggo$ with respect to $B$.  Recall that $B|_{\kg\times\kg}<0$ since it is well known that $\overline{\Ad(K)}$ is compact in $\Gl(\ggo)$ and the isotropy representation $\ad:\kg\longrightarrow\End(\pg)$ is faithful by almost-effectiveness .  This implies that $\kg\cap\pg=0$, and since $\dim{\pg}\geq\dim{\ggo}-\dim{\kg}$ we obtain $\ggo=\kg+\pg$, concluding the proof.
\end{proof}

\subsection{Varying Lie brackets viewpoint}\label{varhs} (See \cite{spacehm} for further information).
Let us fix for the rest of the section a $(q+n)$-dimensional real vector space $\ggo$ together with a direct sum decomposition
\begin{equation}\label{fixdec}
\ggo=\kg\oplus\pg, \qquad \dim{\kg}=q, \qquad \dim{\pg=n},
\end{equation}
and an inner product $\ip$ on $\pg$.  We consider the space of all skew-symmetric algebras (or brackets) of dimension $q+n$, which is
parameterized by the vector space
$$
V_{q+n}:=\{\mu:\ggo\times\ggo\longrightarrow\ggo : \mu\; \mbox{bilinear and
skew-symmetric}\},
$$
and we set
$$
V_{n}:=\{\mu:\pg\times\pg\longrightarrow\pg : \mu\; \mbox{bilinear and
skew-symmetric}\}.
$$

\begin{definition}\label{hqn} The subset $\hca_{q,n}\subset V_{q+n}$ consists of the brackets $\mu\in V_{q+n}$ such that:
\begin{itemize}
\item [(h1)]  $\mu$ satisfies the Jacobi condition, $\mu(\kg,\kg)\subset\kg$ and $\mu(\kg,\pg)\subset\pg$.

\item[(h2)] If $G_\mu$ denotes the simply connected Lie group with Lie algebra $(\ggo,\mu)$ and $K_\mu$ is the connected Lie subgroup of $G_\mu$ with Lie algebra $\kg$, then $K_\mu$ is closed in $G_\mu$.

\item[(h3)] $\ip$ is $\ad_{\mu}{\kg}$-invariant (i.e. $(\ad_{\mu}{Z}|_{\pg})^t=-\ad_{\mu}{Z}|_{\pg}$ for all $Z\in\kg$).

\item[(h4)] $\{ Z\in\kg:\mu(Z,\pg)=0\}=0$.
\end{itemize}
\end{definition}

Each $\mu\in\hca_{q,n}$ defines a unique simply connected homogeneous space,
\begin{equation}\label{hsmu}
\mu\in\hca_{q,n}\rightsquigarrow\left(G_{\mu}/K_{\mu},g_\mu\right),
\end{equation}
with reductive decomposition $\ggo=\kg\oplus\pg$ and $g_\mu(o_\mu)=\ip$, where $o_\mu:=e_\mu K_\mu$ is the origin of $G_\mu/K_\mu$ and $e_\mu\in G_\mu$ is the identity element.  It is almost-effective by (h4), and it follows from (h3) that $\ip$ is $\Ad(K_{\mu})$-invariant as $K_{\mu}$ is connected.  We note that any $n$-dimensional simply connected homogeneous space $(G/K,g)$ which is almost-effective can be identified with some $\mu\in\hca_{q,n}$, where $q=\dim{K}$.  Indeed, $G$ can be assumed to be simply connected without losing almost-effectiveness, and we can identify any reductive decomposition with $\ggo=\kg\oplus\pg$. In this way, $\mu$ will be precisely the Lie bracket of $\ggo$.

We also fix from now on a basis $\{ Z_1,\dots,Z_q\}$ of $\kg$ and an orthonormal basis $\{ X_1,\dots,X_n\}$ of $\pg$ (see (\ref{fixdec})) and use them to identify the groups $\Gl(\ggo)$, $\Gl(\kg)$, $\Gl(\pg)$ and $\Or(\pg,\ip)$, with $\Gl_{q+n}(\RR)$, $\Gl_q(\RR)$, $\Gl_n(\RR)$ and $\Or(n)$, respectively.

There is a natural linear action of $\Gl_{q+n}(\RR)$ on $V_{q+n}$ given by
\begin{equation}\label{action}
h\cdot\mu(X,Y)=h\mu(h^{-1}X,h^{-1}Y), \qquad X,Y\in\ggo, \quad h\in\Gl_{q+n}(\RR),\quad \mu\in V_{q+n}.
\end{equation}

If $\mu\in\hca_{q,n}$, then $h\cdot\mu\in\hca_{q,n}$ for any $h\in\Gl_{q+n}(\RR)$ of the form
\begin{equation}\label{formh}
h:=\left[\begin{smallmatrix} h_q&0\\ 0&h_n
\end{smallmatrix}\right]\in\Gl_{q+n}(\RR), \quad h_q\in\Gl_q(\RR), \quad h_n\in\Gl_n(\RR),
\end{equation}
such that
\begin{equation}\label{adkh}
[h_n^th_n,\ad_{\mu}{\kg}|_{\pg}]=0.
\end{equation}
We have that $\left(G_{h\cdot\mu}/K_{h\cdot\mu},g_{h\cdot\mu}\right)$ is equivariantly isometric to $\left(G_{\mu}/K_{\mu},g_{\la h_n\cdot,h_n\cdot\ra}\right)$, and in particular, the subset
$$
\left\{ h\cdot\mu:h_q=I,\, h_n\,\mbox{satisfies (\ref{adkh})}\right\}\subset\hca_{q,n},
$$
parameterizes the set of all $G_\mu$-invariant metrics on $G_\mu/K_\mu$.  Also, by setting $h_q=I$, $h_n=\frac{1}{c}I$, $c\ne 0$, we get the rescaled $G_\mu$-invariant metric $\frac{1}{c^2}g_{\ip}$ on $G\mu/K_\mu$, which is isometric to the element of $\hca_{q,n}$ denoted by $c\cdot\mu$ and defined by
\begin{equation}\label{scmu}
c\cdot\mu|_{\kg\times\kg}=\mu, \qquad c\cdot\mu|_{\kg\times\pg}=\mu, \qquad c\cdot\mu|_{\pg\times\pg}=c^2\mu_{\kg}+c\mu_{\pg},
\end{equation}
where the subscripts denote the $\kg$ and $\pg$-components of $\mu|_{\pg\times\pg}$ given by
\begin{equation}\label{decmu}
\mu(X,Y)=\mu_{\kg}(X,Y)+\mu_{\pg}(X,Y), \qquad \mu_{\kg}(X,Y)\in\kg, \quad \mu_{\pg}(X,Y)\in\pg, \qquad\forall X,Y\in\pg.
\end{equation}
The $\RR^*$-action on $\hca_{q,n}$, $\mu\mapsto c\cdot\mu$, can therefore be considered as a geometric rescaling of the homogeneous space $(G_\mu/K_\mu,g_\mu)$.

\subsection{Homogeneous Ricci flow}\label{hrf} (See \cite[Section 3]{homRF} for a more detailed treatment).
Let $(M,g_0)$ be a simply connected homogeneous manifold.  Thus $(M,g_0)$ has a presentation as a homogeneous space of the form $\left(G_{\mu_0}/K_{\mu_0},g_{\mu_0}\right)$ for some $\mu_0\in\hca_{q,n}$, with reductive decomposition $\ggo=\kg\oplus\pg$ (see Section \ref{varhs}).  Let $g(t)$ be the unique homogeneous solution to a {\it normalized Ricci flow}
\begin{equation}\label{RFrn}
\dpar g(t)=-2\ricci(g(t))-2r(t)g(t),\qquad g(0)=g_0,
\end{equation}
for some {\it normalization function} $r(t)$ which may depend on $g(t)$.  For any continuous function $r$, $g(t)$ can be obtained by just rescaling and reparameterizing the time variable of the usual {\it unnormalized} (i.e. $r\equiv 0$) Ricci flow solution.  It follows that $g(t)$ is $G_{\mu_0}$-invariant for all $t$, and thus $(M,g(t))$ is isometric to the homogeneous space $\left(G_{\mu_0}/K_{\mu_0},g_{\ip_t}\right)$, where $\ip_t:=g(t)(o_{\mu_0})$ is a family of inner products on $\pg$.  The Ricci flow equation (\ref{RFrn}) is therefore equivalent to the ODE system
\begin{equation}\label{RFiprn}
\ddt\ip_t=-2\ricci(\ip_t)-2r(t)\ip_t, \qquad \ip_0=\ip,
\end{equation}
where $\ricci(\ip_t):=\ricci(g(t))(o_{\mu_0})$.

\subsection{The bracket flow}\label{lbflow} (See \cite[Sections 3.1-3.3]{homRF} for a more gentle presentation).
The ODE system for a family $\mu(t)\in V_{q+n}=\lamg$
of bilinear and skew-symmetric maps defined by
\begin{equation}\label{BFrn}
\ddt\mu=-\pi\left(\left[\begin{smallmatrix} 0&0\\ 0&\Ricci_{\mu}+rI
\end{smallmatrix}\right]\right)\mu, \qquad \mu(0)=\mu_0,
\end{equation}
is called a {\it normalized bracket flow}. Here $\Ricci_{\mu}$ is defined as in \cite[Section 2.3]{homRF} and coincides with the Ricci operator when $\mu\in\hca_{q,n}$, and $\pi:\glg_{q+n}(\RR)\longrightarrow\End(V_{q+n})$ is the natural representation given by
\begin{equation}\label{actiong}
\pi(A)\mu=A\mu(\cdot,\cdot)-\mu(A\cdot,\cdot)-\mu(\cdot,A\cdot),
\qquad A\in\glg_{q+n}(\RR),\quad\mu\in V_{q+n}.
\end{equation}
We note that $\pi$ is the derivative of the $\Gl_{q+n}(\RR)$-representation defined in (\ref{action}).

A homogeneous space $(G_{\mu(t)}/K_{\mu(t)},g_{\mu(t)})$ can indeed
be associated to each $\mu(t)$ in a bracket flow solution provided that $\mu_0\in\hca_{q,n}$, since it follows that $\mu(t)\in\hca_{q,n}$ for all $t$.  For a given simply connected homogeneous manifold $(M,g_0)=\left(G_{\mu_0}/K_{\mu_0},g_{\mu_0}\right)$, $\mu_0\in\hca_{q,n}$, we can therefore consider the following one-parameter families:
\begin{equation}\label{3rmrn}
(M,g(t)), \qquad \left(G_{\mu_0}/K_{\mu_0},g_{\ip_t}\right), \qquad \left(G_{\mu(t)}/K_{\mu(t)},g_{\mu(t)}\right),
\end{equation}
where $g(t)$, $\ip_t$ and $\mu(t)$ are the solutions to the normalized Ricci flows (\ref{RFrn}), (\ref{RFiprn}) and the normalized bracket flow (\ref{BFrn}), respectively.  Recall that $\ggo=\kg\oplus\pg$ is a reductive decomposition for any of the homogeneous spaces involved.   According to the following result, the Ricci flow and the bracket flow are intimately related.

\begin{theorem}\label{eqflrn}\cite[Theorem 3.10]{homRF}
There exist diffeomorphisms $\vp(t):M\longrightarrow G_{\mu(t)}/K_{\mu(t)}$ such that
$$
g(t)=\vp(t)^*g_{\mu(t)}, \qquad t\in (T_-,T_+),
$$
where $-\infty\leq T_-<0<T_+\leq\infty$ and $(T_-,T_+)$ is the maximal interval of time existence for both flows.  Moreover, if we identify $M=G_{\mu_0}/K_{\mu_0}$, then $\vp(t):G_{\mu_0}/K_{\mu_0}\longrightarrow G_{\mu(t)}/K_{\mu(t)}$ can be chosen as the equivariant diffeomorphism determined by the Lie group isomorphism between $G_{\mu_0}$ and $G_{\mu(t)}$ with derivative $\tilde{h}:=\left[\begin{smallmatrix} I&0\\ 0&h \end{smallmatrix}\right]:\ggo\longrightarrow\ggo$, where $h(t)=d\vp(t)|_{o_{\mu_0}}:\pg\longrightarrow\pg$ is the solution to any of the following ODE systems:

\begin{itemize}
\item[(i)] $\ddt h=-h(\Ricci(\ip_t)+r(t)I)$, $\quad h(0)=I$.

\item[(ii)] $\ddt h=-(\Ricci_{\mu(t)}+r(t)I)h$, $\quad h(0)=I$.
\end{itemize}
The following conditions also hold:
\begin{itemize}
\item[(iii)] $\ip_t=\la h\cdot,h\cdot\ra$.

\item[(iv)] $\mu(t)=\tilde{h}\mu_0(\tilde{h}^{-1}\cdot,\tilde{h}^{-1}\cdot)$.
\end{itemize}
\end{theorem}

It follows that $\mu(t)|_{\kg\times\ggo}=\mu_0|_{\kg\times\ggo}$ for all $t\in (T_-,T_+)$, that is, only $\mu(t)|_{\pg\times\pg}$ is actually evolving, and so the bracket flow equation (\ref{BFrn}) can be rewritten as the system
\begin{equation}\label{BFrnsis}
\left\{\begin{array}{ll}
\ddt\mu_{\kg}=\mu_{\kg}(\Ricci_{\mu}\cdot,\cdot)+\mu_{\kg}(\cdot,\Ricci_{\mu}\cdot) +2r\mu_{\kg}(\cdot,\cdot), & \\
& \mu_{\kg}(0)+\mu_{\pg}(0)=\mu_0|_{\pg\times\pg},\\
\ddt\mu_{\pg}=-\pi_n(\Ricci_{\mu}+rI)\mu_{\pg} =-\pi_n(\Ricci_{\mu})\mu_{\pg} +r\mu_{\pg}. &
\end{array}\right.
\end{equation}
where $\mu_{\kg}$ and $\mu_{\pg}$ are the components of $\mu|_{\pg\times\pg}$ as in (\ref{decmu}) and $\pi_n:\glg_n(\RR)\longrightarrow\End(V_n)$ is the  representation defined in (\ref{actiong}) for $q=0$.

Let $\nu(t)$ denote from now on the unnormalized (i.e. $r\equiv 0$) bracket flow solution with $\nu(0)=\mu_0$.  Then any normalized bracket flow solution is given by
\begin{equation}\label{ctau}
\mu(t)=c(t)\cdot\nu(\tau(t)), \qquad t\in (T_-,T_+),
\end{equation}
for some rescaling $c(t)>0$ (defined by $c'=rc$, $c(0)=1$) and time reparameterization $\tau(t)$ (defined by $\tau'=c^2$, $\tau(0)=0$).  If $(T_-^0,T_+^0)$ denotes the maximal interval of time existence for $\nu(t)$, then $\tau:(T_-,T_+)\longrightarrow(T_-^0,T_+^0)$ is a strictly increasing function, though non-necessarily surjective.

\section{Homogeneous Ricci solitons}\label{hrs}

A Riemannian manifold $(M,g)$ is called {\it Einstein} if $\ricci(g)=cg$ for some $c\in\RR$ (see e.g. \cite{Bss}).  The Einstein equation for an $n$-dimensional homogeneous space is just a system of
$\tfrac{n(n+1)}{2}$ algebraic equations, but unfortunately, quite an involved one,
and it is still open the question of which homogeneous spaces $G/K$ admit a $G$-invariant Einstein metric (see the surveys \cite{Wng, cruzchica} and the references therein).  In the noncompact homogeneous case, the only known non-flat examples until now are all simply connected {\it solvmanifolds}, which are defined in this paper to be solvable Lie groups endowed with a left-invariant metric. According to the long standing {\it Alekseevskii conjecture} (see \cite[7.57]{Bss} and \cite{alek}), asserting that any Einstein connected homogeneous manifold of negative scalar curvature is diffeomorphic to a Euclidean space,  Einstein solvmanifolds might exhaust all the possibilities for noncompact homogeneous Einstein manifolds.

A nice and important generalization of Einstein metrics is the following notion.  A complete Riemannian manifold $(M,g)$ is called a {\it Ricci
soliton} if
\begin{equation}\label{rseq}
\ricci(g)=cg+\lca_Xg, \qquad\mbox{for some}\; c\in\RR, \quad X\in\chi(M),
\end{equation}
where $\lca_Xg$ is the usual Lie derivative of $g$ in the direction of the (complete) vector field $X$ (as in the Einstein case, $c$ is often called the {\it cosmological constant} of the Ricci soliton $g$).  Ricci solitons
correspond to solutions of the Ricci flow that evolve self similarly, that is, only
by scaling and pullback by diffeomorphisms, and often arise as limits of dilations
of singularities of the Ricci flow.  More precisely, $g$ is a Ricci soliton if and
only if the one-parameter family of metrics
\begin{equation}\label{rssol}
g(t)=(-2ct+1)\vp_t^*g,
\end{equation}
is a solution to the Ricci flow for some one-parameter group $\vp_t$ of diffeomorphisms of $M$ (see e.g. \cite{libro,Cao} and the references therein for further information on Ricci solitons).

From results due to Ivey, Naber, Perelman and Petersen-Wylie (see \cite[Section 2]{solvsolitons}), it follows that any
{\it nontrivial} (i.e. non-Einstein and not the product of an Einstein homogeneous manifold with a Euclidean space) homogeneous Ricci soliton must be noncompact, expanding (i.e. $c<0$) and non-gradient (i.e. $X$ is not the gradient field of any smooth function on $M$).  Any known example so far of a nontrivial homogeneous Ricci soliton is isometric to a simply connected {\it solvsoliton}, that is, a solvmanifold $(S,g)$ satisfying
\begin{equation}\label{rsD}
\Ricci(g)=cI+D, \qquad\mbox{for some}\quad c\in\RR,\quad D\in\Der(\sg),
\end{equation}
once the metric $g$ is identified with an inner product on the Lie algebra $\sg$ of $S$.  When $S$ is nilpotent, these metrics are called {\it nilsolitons}
and are precisely the nilpotent parts of Einstein solvmanifolds (see the survey \cite{cruzchica} for further information).  It is proved in \cite{solvsolitons} that, up to isometry, any solvsoliton can be
obtained via a very simple construction from a nilsoliton $(N,g_1)$ together with any
abelian Lie algebra of symmetric derivations of the metric Lie algebra of $(N,g_1)$.
Furthermore, a given solvable Lie group can admit at most one solvsoliton left
invariant metric up to isometry and scaling, and another consequence of
\cite{solvsolitons} is that any Ricci soliton obtained via (\ref{rsD}) is necessarily
simply connected (see also \cite{Lfn,Wll} for examples and classification results on solvsolitons).

The following recently proved result completes the picture for Ricci soliton left-invariant metrics on solvable Lie groups.

\begin{theorem} \cite[Theorem 1.1]{Jbl}
Any (nonflat) Ricci soliton admitting a transitive solvable Lie group of isometries is isometric to a simply connected solvsoliton.
\end{theorem}

The concept of solvsoliton can be easily generalized to the class of all homogeneous spaces as follows.

\begin{definition}\label{as}
A homogeneous space $(G/K,g_{\ip})$ with reductive decomposition $\ggo=\kg\oplus\pg$ is said to be an {\it algebraic
soliton} if there exist $c\in\RR$ and $D\in\Der(\ggo)$ such that $D\kg\subset\kg$ and
$$
\Ricci(g_{\ip})=cI+D_{\pg},
$$
where $D_{\pg}:=\proy\circ D|_{\pg}$ and $\proy:\ggo=\kg\oplus\pg\longrightarrow\pg$ is the linear projection.
\end{definition}

We note that Einstein homogeneous manifolds are algebraic solitons with respect to any presentation as a homogeneous space and any reductive decomposition, by just taking $D=0$.

The following result supports in a way the above definition.

\begin{proposition}\label{asrs}
Any simply connected algebraic soliton $(G/K,g_{\ip})$ is a Ricci soliton.
\end{proposition}

\begin{remark}
The hypothesis of $G/K$ being simply connected is in general necessary in the above proposition, as the case of solvsolitons shows (see \cite[Remark 4.12]{solvsolitons}).
\end{remark}

\begin{proof}
We can assume that $G$ is simply connected and still have that $G/K$ is almost-effective.  Notice that $K$ is therefore connected as $G/K$ is simply connected. Since $D\in\Der(\ggo)$ we have that $e^{tD}\in\Aut(\ggo)$ and thus there exists
$\tilde{\vp}_t\in\Aut(G)$ such that $d\tilde{\vp}_t|_e=e^{tD}$ for all $t\in\RR$.  By using that $K$ is connected and $D\kg\subset\kg$, it is easy to see that
$\tilde{\vp}_t(K)=K$ for all $t$.  This implies that
$\tilde{\vp}_t$ defines a diffeomorphism $\vp_t$ of $M=G/K$ by $\vp_t(uK)=\tilde{\vp}_t(u)K$ for any
$u\in G$, which therefore satisfies at the origin that $d\vp_t|_{o}=e^{tD_{\pg}}$. Let  $X_D$ denote the vector field of $M$ defined by the one-parameter subgroup
$\{\vp_t\}\subset\Diff(M)$, that is, $X_D(p)=\ddt|_0\vp_t(p)$ for
any $p\in M$. It follows from the symmetry of $D_{\pg}$ that
\begin{equation}\label{Lder}
\lca_{X_D}g_{\ip}=\ddt|_0\vp^*_tg_{\ip} =\ddt|_0\la e^{-tD_{\pg}}\cdot,
e^{-tD_{\pg}}\cdot\ra = -2\la D_{\pg}\cdot,\cdot\ra,
\end{equation}
but since $\Ricci=cI+D_{\pg}$, we obtain that $\ricci(g_{\ip})=cg_{\ip}-\unm \lca_{X_D}g_{\ip}$, and so $g_{\ip}$ is a Ricci soliton (see (\ref{rseq})), as was to be shown.
\end{proof}

\begin{remark}\label{Dkcero}
In Definition \ref{as}, $D\kg=0$ necessarily holds. Indeed, we have that
$$
\ad{DZ}|_\pg=[D|_\pg,\ad{Z}|_\pg]=[\Ricci(g_{\ip}),\ad{Z}|_\pg]=0, \qquad\forall Z\in\kg,
$$
and thus $D\kg=0$ by almost-effectiveness.
\end{remark}

From the proof of Proposition \ref{asrs}, one may perceive that there is a more general way to consider a homogeneous Ricci soliton `algebraic', in the sense that the algebraic structure of some of its presentations as a homogeneous space be strongly involved.

\begin{definition}\label{sas}\cite[Definition 1.4]{Jbl}
A homogeneous space $(G/K,g)$ is called a {\it semi-algebraic
soliton} if there exists a one-parameter family $\tilde{\vp}_t\in\Aut(G)$ with $\tilde{\vp}_t(K)=K$ such that
$$
g(t)=c(t)\vp_t^*g,
$$
is a solution to the (unnormalized) Ricci flow equation (\ref{RFrn}) starting at $g(0)=g$ for some scaling function $c(t)>0$, where $\vp_t\in\Diff(G/K)$ is the diffeomorphism determined by $\tilde{\vp}_t$.
\end{definition}

As the following example shows, a homogeneous Ricci soliton is not always semi-algebraic with respect to a given presentation as a homogenous space.

\begin{example}\label{nosemi1}
The direct product $S=S_1\times S_2$ of a completely solvable nonflat solvsoliton $S_1$ and a flat nonabelian solvmanifold $S_2$ is a Ricci soliton which is not semi-algebraic when presented as a left-invariant metric on $S$ (see \cite[Example 1.3]{Jbl}).
\end{example}

The following result confirms the protagonism of the algebraic side of homogeneous manifolds in regard to Ricci soliton theory.

\begin{theorem}\cite[Proposition 2.2]{Jbl}\label{semi}
Any homogeneous Ricci soliton $(M,g)$ is a semi-algebraic soliton with respect to its full
isometry group $G=\Iso(M,g)$.  Moreover, $\tilde{\vp}_t$ can be chosen to be a one-parameter subgroup of $\Aut(G)$ such that $\tilde{\vp}_t|_{K_0}=id$, where $K_0$ is the identity component of $K$.
\end{theorem}

It follows from (\ref{Lder}) that if $\left(G/K,g_{\ip}\right)$ is a semi-algebraic soliton with reductive decomposition $\ggo=\kg\oplus\pg$, then
\begin{equation}\label{alg2}
\Ricci(g_{\ip})=cI+\unm\left(D_{\pg}+D_{\pg}^t\right), \qquad\mbox{for some} \quad c\in\RR, \quad D\in\Der(\ggo), \quad D\kg=0,
\end{equation}
where actually $D=\ddt|_0\tilde{\vp}_t$ (see also \cite[Proposition 2.3]{Jbl}).  Conversely, if condition (\ref{alg2}) holds for some reductive decomposition and $G/K$ is simply connected, then one can prove in much the same way as Proposition \ref{asrs} that $\left(G/K,g_{\ip}\right)$ is indeed a Ricci soliton with $\ricci(g_{\ip})=cg_{\ip}-\unm \lca_{X_D}g_{\ip}$.

\begin{example}\label{nosemi2}
The following example was generously provided to us by M. Jablonski and it is given by the $6$-dimensional solvmanifold whose metric Lie algebra has an orthonormal basis $\{X_1,Y_1,Z_1,X_2,Y_2,Z_2\}$ with Lie bracket
$$
[X_1,Y_1]=Z_1, \quad [X_1,X_2]=Y_2, \quad [X_1,Y_2]=-X_2, \quad [X_2,Y_2]=Z_2.
$$
It is easy to see that it is not a semi-algebraic soliton.  Indeed, the Ricci operator is $\Ricci=\diag(-\unm,-\unm,\unm,-\unm,-\unm,\unm)$ and if (\ref{alg2}) holds, then $c=-\unm$ since $D$ must leave the nilradical $\la X_2,\dots,Z_2\ra$ invariant.  Now, restricted to the Lie ideal $\la X_2,Y_2,Z_2\ra$, the diagonal part of $D$ has the form $\diag(a,b,a+b)$, and thus $a=b=0$ and $-\unm+a+b=\unm$, a contradiction.  However, it is easy to prove that it is isometric to the nilsoliton $H_3\times H_3$, where $H_3$ denotes the $3$-dimensional Heisenberg group.  The bracket flow evolution of this solvmanifold will be analyzed in Section \ref{evol-nosemi}.
\end{example}

The special reductive decomposition given in Lemma \ref{Bkp0} establishes some constraints on the behavior of derivations, and consequently on the structure of semi-algebraic solitons.

\begin{lemma}\label{Dkp}
Let $(G/K,g_{\ip})$ be a homogeneous space with reductive decomposition $\ggo=\kg\oplus\pg$, and assume in addition that $B(\kg,\pg)=0$, where $B$ is the Killing form of $\ggo$.  If $D \in \Der(\ggo)$ satisfies $D \kg \subset \kg$, then $D\pg \subset \pg$.
\end{lemma}

\begin{proof}
For $X\in \pg$, $Z = DX$, we write $Z = Z_\kg + Z_\pg$ according to the decomposition $\ggo=\kg\oplus\pg$.  By using that $D \kg \subset \kg$ and $B(\kg,\pg)=0$, we get
\[
0 = B(Z, Z_\kg) + B(X, DZ_\kg) = B(Z_\kg, Z_\kg),
\]
which implies that $Z_\kg =0$ since $B|_{\kg \times \kg}$ is negative definite.
\end{proof}

\begin{corollary}\label{semiB}
Let $(G/K,g_{\ip})$ be a semi-algebraic soliton with reductive decomposition $\ggo=\kg\oplus\pg$, and assume in addition that $B(\kg,\pg)=0$.  Then,
$$
\Ricci(g_{\ip})=cI+\unm\left(D_{\pg}+D_{\pg}^t\right), \qquad\mbox{with} \qquad D = \left[\begin{smallmatrix} 0&0\\ 0&D_\pg\\\end{smallmatrix}\right]\in\Der(\ggo).
$$
\end{corollary}

\begin{example}\label{Dpnop}
It may happen that $\pg$ is not preserved by the derivation $D$. Consider for instance any nilsoliton $(\ngo,\ip)$, say with $\Ricci_\ngo = cI + D_0$, $D_0 \in \Der(\ngo)$.  Assume there exists a nonzero $B\in \Der(\ngo)\cap\sog(\ngo,\ip)$.  This gives rise to a presentation of the nilsoliton as a homogeneous space with one-dimensional isotropy and reductive decomposition $\ggo = \kg \oplus \ngo$, where $\kg = \RR B$ and $B$ is acting on $\ngo$ as usual.  Now suppose that there is a nonzero $X\in\ngo$ such that $BX=0$, and let $\pg \subset \ggo$ be the subspace spanned by $\alpha$, where $\alpha$ is the set obtained by replacing $X$ with $X+B$ in an orthonormal basis $\beta$ of $(\ngo,\ip)$ containing $X$.  We see that $\ggo = \kg \oplus \pg$ is also a reductive decomposition. Under the natural identification of $\pg$ with $\ngo$, the basis $\alpha$ and $\beta$ turn out to be identified, and then the Ricci operator $\Ricci_\pg$ (according to $\ggo = \kg \oplus \pg$) is given by
\[
\Ricci_\pg = cI + D_1, \quad D_1\in \End(\pg), \quad [D_1]_\alpha = [D_0]_\beta,
\]
which implies that it is still an algebraic soliton with respect to the reductive decomposition $\ggo=\kg\oplus\pg$.  Now assume that $D_1 = \unm(D_\pg + D_\pg^t)$, for some $D:= \left[\begin{smallmatrix}0 & 0 \\ 0 & D_\pg \end{smallmatrix}\right] \in \Der(\ggo)$. Then, by using that $D \pg \subset \pg \cap \ngo$ (in fact, $D\ggo \subset \ngo$ since $\ggo$ is solvable and $\ngo$ is its nilradical) and $\pg \cap \ngo \perp (X+B)$, we obtain that
\[
\la D_0 X, X \ra = \la D_1 (X+B), (X+B)\ra  = \la D (X+B), (X+B)\ra = 0,
\]
which contradicts the fact that $D_0$ is positive definite (see e.g. \cite[Section 2]{cruzchica}).  Thus this reductive decomposition does not allow us to present this nilsoliton as a semi-algebraic soliton with a derivation leaving $\pg$ invariant.  The bracket flow evolution of this example in the case $\ngo$ is the 3-dimensional Heisenberg Lie algebra will be studied in Section \ref{evol-Dpnop}.
\end{example}

In \cite{alek}, some structural results on Lie theoretical aspects of semi-algebraic solitons are given.  In the case of nilmanifolds, or more precisely, when $G$ is nilpotent and simply connected and $K$ is trivial, condition (\ref{alg2}) is equivalent to $(G,\ip)$ being a nilsoliton, since $D^t$ turns out to be a derivation in this case.  Indeed, if $\Ricci(\ip)=cI+D+D^t$, then by using \cite[(19)]{homRF} one obtains that
$$
0=\tr{\Ricci(\ip)[D,\Ricci(\ip)]}=\tr{\Ricci(\ip)[D,D^t]}=\unc\la\pi(D^t)\lb,\pi(D^t)\lb\ra,
$$
and hence $D^t\in\Der(\ggo)$.   This fixes a gap in the proof of \cite[Proposition 1.1]{soliton}, which was kindly pointed out to the second author by M. Jablonski.

\begin{figure}\label{HRSfig}
    \includegraphics[width=\textwidth]{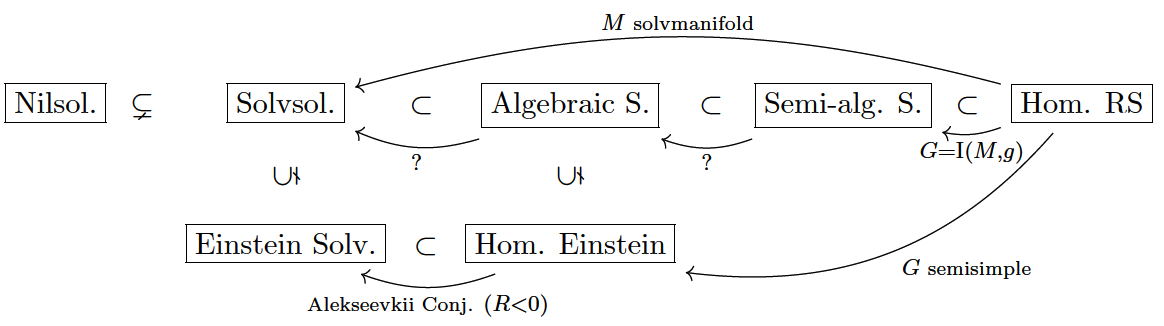}
\caption{Homogeneous Ricci solitons}
\end{figure}


It is proved in \cite[Theorem 1.6]{Jbl} that any homogeneous Ricci soliton admitting a transitive semi-simple group of isometries must be Einstein.

As far as we know, beyond solvmanifolds, the current status of knowledge on simply connected (nontrivial) homogeneous Ricci solitons can be summarized as in Figure \ref{HRSfig}. Recall that the only known examples for now are all isometric to solvsolitons, i.e. left-invariant algebraic solitons on solvable Lie groups.

\section{Algebraic solitons and the bracket flow}\label{ASBF}

We study in this section how semi-algebraic solitons evolve according to the bracket flow.  Algebraic solitons are proved to be the only possible limits, backward and forward, for any bracket flow solution.  This fact and the equivalence between the bracket and Ricci flows (see Theorem \ref{eqflrn}) suggest that algebraic solitons might exhaust the class of all homogeneous Ricci solitons (up to isometry).

\begin{proposition}\label{limrs}
Let $\mu(t)$ be a solution to any normalized bracket flow (see {\rm (\ref{BFrn})} or {\rm (\ref{BFrnsis})}).
\begin{itemize}
\item[(i)] If $\mu_0$ is a fixed point (i.e. $\mu(t)\equiv\mu_0$), then $\left(G_{\mu_0}/K_{\mu_0},g_{\mu_0}\right)$ is an algebraic soliton with $\Ricci_{\mu_0}=cI+D_{\pg}$, for some $c\in\RR$ and such that $\left[\begin{smallmatrix} 0&0\\ 0&D_{\pg} \end{smallmatrix}\right]\in\Der(\ggo,\mu_0)$.


\item[(ii)] If $\mu(t)\to\lambda\in\hca_{q,n}$, as $t\to T_{\pm}$, then $T_{\pm}=\pm\infty$ and $\left(G_\lambda/K_\lambda,g_\lambda\right)$ is an algebraic soliton as in part (i).

\item[(iii)] Assume that $\mu(t)\to\lambda\in\hca_{q,n}$, as $t\to\pm\infty$, with $\lambda|_{\pg\times\pg}\ne 0$.  Then the limit $\tilde{\lambda}$ of any other $\tilde{r}$-normalized bracket flow solution necessarily satisfies $\tilde{\lambda}=c\cdot\lambda$ for some $c\geq 0$.  In particular, $\tilde{\lambda}$ is either flat ($c=0$) or homothetic to $\lambda$ ($c>0$).
\end{itemize}
\end{proposition}

\begin{proof}
Let $\lambda$ be a fixed point for some normalized bracket flow of the form (\ref{BFrn}) (i.e. $\mu(t)\equiv\mu_0=\lambda$).  Thus
$$
-\pi\left(\left[\begin{smallmatrix} 0&0\\ 0&\Ricci_{\lambda}+r(0)I
\end{smallmatrix}\right]\right)\lambda=0,
$$
from which we deduce that $\left(G_\lambda/K_\lambda,g_\lambda\right)$ is an algebraic soliton with $c=-r(0)$ and $D=\left[\begin{smallmatrix} 0&0\\ 0&\Ricci_{\lambda}+r(0)I\end{smallmatrix}\right]$ (see {\rm Definition \ref{as}}).  This proves parts (i) and (ii).

Let us now prove (iii).  It follows from \cite[Lemma 3.9]{homRF} that both time reparametrizations $\tau(t)$ and $\tilde{\tau}(t)$ converge to $T^0_{\pm}$, as $t\to\pm\infty$ (see (\ref{ctau})).  We project the curves determined by the two bracket flow solutions onto the quotient $V_{q+n}/\RR_{>0}$, where the $\RR_{>0}$-action is given by the rescaling (\ref{scmu}), and denote by $[v]$ the equivalence class of a vector $v\in V_{q+n}$.  By (\ref{ctau}), we obtain that $[\nu(s)]$ converges to both $[\lambda]$ and $[\tilde{\lambda}]$, as $s\to T_{\pm}^0$, relative to the quotient topology.  As the pairs of points which can not be separated by disjoint open sets in $V_{q+n}/\RR_{>0}$ are all of the form $[v]$, $[0\cdot v]$ for some $v\in V_{q+n}$, either $\tilde{\lambda}=0\cdot\lambda$ or $[\tilde{\lambda}]=[\lambda]$, as was to be shown.
\end{proof}

As an application of Theorem \ref{eqflrn}, unnormalized Ricci and bracket flow solutions are proved in what follows to have a very simple form for semi-algebraic and algebraic solitons.

\begin{proposition}\label{saBF}
Let $\left(G_{\mu_0}/K_{\mu_0},g_{\mu_0}\right)$, $\mu_0 \in \hca_{q,n}$, be a homogeneous space that is a semi-algebraic soliton, say with
$$
\Ricci_{\mu_0}=cI+\unm(D_\pg+D_\pg^t), \qquad c\in\RR, \qquad D:=\left[\begin{smallmatrix} 0&\ast\\ 0&D_\pg \end{smallmatrix}\right] \in\Der(\ggo,\mu_0).
$$
Then the unnormalized bracket flow solution to \eqref{BFrn} (i.e. with $r\equiv 0$) is given by
\begin{equation}\label{saevol}
\nu(t) = (-2ct+1)^{-1/2} \cdot \left[\begin{smallmatrix} I&0\\ 0& e^{s(t)A} e^{-s(t)D_\pg} \end{smallmatrix}\right] \cdot \mu_0, \qquad t\in\left\{\begin{array}{lcl} (\tfrac{1}{2c},\infty), && c<0, \\ (-\infty,\tfrac{1}{2c}), && c>0, \\ (-\infty,\infty), && c=0, \end{array}\right.
\end{equation}
where $ A = \unm(D_\pg - D_\pg^t)$ and $s(t) = -\tfrac{1}{2c} \ln (-2ct+1) $ (for $c=0$, $s(t)\equiv 1$).

Conversely, if the unnormalized bracket flow $\nu(t)$ evolves as in \eqref{saevol}, then $\left(G_{\mu_0}/K_{\mu_0},g_{\mu_0}\right)$ is a semi-algebraic soliton.  The corresponding derivation $\tilde{D}$ satisfies that $A = \unm(\tilde{D}_\pg - \tilde{D}_\pg^t)$, although possibly $\tilde{D}_{\pg}\ne D_\pg$.
\end{proposition}

\begin{proof}
By using equality $\Ricci_{\mu_0}=cI+\unm(D_\pg+D_\pg^t)$, it is easy to check that
\[
\ip_t = (-2ct+1) \la e^{-s(t)D_\pg}  \cdot  ,  e^{-s(t)D_\pg}  \cdot \ra,
\]
is a solution to the unnormalized Ricci flow \eqref{RFiprn}.  The Ricci operator of $\ip_t$ is therefore given by
\[
\Ricci(\ip_t) = (-2ct+1)^{-1} e^{s(t) D_\pg} \Ricci(\ip) e^{-s(t) D_\pg} = (-2ct+1)^{-1} (cI + D_\pg - e^{s(t) D_\pg} A e^{-s(t) D_\pg} )
\]
(recall that $\Ricci(\ip) = \Ricci_{\mu_0} = cI + D_\pg - A$). Now we solve the differential equation given in part (i) of Theorem \ref{eqflrn}, getting $h(t) = (-2ct+1)^{1/2} e^{s(t) A} e^{-s(t) D_\pg}  $, and by using part (iv) of the same theorem we obtain the desired formula for $\nu(t)$.

The converse follows by computing  $\ddt \nu(t)\big|_0$, which equals $-\pi\left(\left[\begin{smallmatrix} 0&0\\ 0&\Ricci_{\mu_0}\end{smallmatrix}\right]\right)\mu_0$.
\end{proof}

\begin{remark}\label{saevol2}
If $D$ has the special form $D =  \left[\begin{smallmatrix} 0&0\\ 0&D_\pg \end{smallmatrix}\right]$ (i.e. $D\pg \subseteq \pg$, or $\ast=0$), which holds for instance if the reductive decomposition satisfies $B_{\mu_0}(\kg,\pg)=0$ for the Killing form $B_{\mu_0}$ of $\mu_0$ (see Corollary \ref{semiB}), then $\left[\begin{smallmatrix} I&0\\ 0&  e^{-s(t)D_\pg} \end{smallmatrix}\right] \in \Aut(\ggo,\mu_0)$, and hence the formula for the unnormalized bracket flow is given by
\begin{equation}\label{saBF2}
\nu(t) = (-2ct+1)^{-1/2} \cdot \left(\left[\begin{smallmatrix} I&0\\ 0& e^{s(t)A} \end{smallmatrix}\right] \cdot \mu_0\right).
\end{equation}
\end{remark}

Algebraic solitons with $D\pg\subset\pg$ are characterized in terms of their bracket flow evolution as follows.

\begin{proposition}\label{rsequiv}
For a homogeneous space $(G_{\mu_0}/K_{\mu_0},g_{\mu_0})$, $\mu_0\in\hca_{q,n}$, the following conditions are equivalent:
\begin{itemize}
\item[(i)] $(G_{\mu_0}/K_{\mu_0},g_{\mu_0})$ is an algebraic soliton with
$$
\Ricci_{\mu_0}=cI+D_\pg, \qquad c\in\RR, \qquad D=\left[\begin{smallmatrix} 0&0\\ 0&D_\pg \end{smallmatrix}\right] \in\Der(\ggo,\mu_0).
$$
\item[(ii)] The unnormalized bracket flow solution is given by
$$
\nu(t)=(-2ct+1)^{-1/2}\cdot\mu_0,
$$
or equivalently,
$$
\nu_{\kg}(t)=(-2ct+1)^{-1}\mu_{\kg}(0), \quad\nu_{\pg}(t)=(-2ct+1)^{-1/2}\mu_{\pg}(0), \quad
t\in\left\{\begin{array}{lcl} (\tfrac{1}{2c},\infty), && c<0, \\ (-\infty,\tfrac{1}{2c}), && c>0, \\ (-\infty,\infty), && c=0. \end{array}\right.
$$
\item[(iii)] The solutions to the unnormalized Ricci flow equations {\rm \eqref{RFrn}} and {\rm \eqref{RFiprn}} are given by
$$
g_{ij}(t)=\la X_i,X_j\ra_t= (-2ct+1)^{r_i/c}\delta_{ij},
$$
where $\{ X_1,\dots,X_n\}$ is an orthonormal basis of $(\pg,\ip)$ of eigenvectors (or Killing vector fields) for $\Ricci_{\mu_0}$ with eigenvalues $\{ r_1,\dots,r_n\}$.
\end{itemize}
\end{proposition}

\begin{proof}
The equivalence between parts (i) and (ii) follows from Proposition \ref{saBF}, by using Remark \ref{saevol2} and that $A=0$ in this case.  To prove that parts (ii) and (iii) are equivalent, we can use Theorem \ref{eqflrn}, as in both cases one obtains
$$
h(t)=e^{a(t)\Ricci_{\mu_0}}, \qquad a(t)=\tfrac{1}{2c}\log(-2ct+1),
$$
concluding the proof of the proposition.
\end{proof}

Part (iii) of the above proposition generalizes results on the asymptotic behavior of some nilsolitons obtained in \cite{Pyn,Wllm} and algebraic solitons on Lie groups in \cite{nicebasis}.  Recall that $c<0$ for any nontrivial homogeneous Ricci soliton.

It follows from Proposition \ref{rsequiv}, (ii) that if $\mu_{\kg}=0$ (e.g. if $q=0$) or $\mu_{\pg}=0$, then the trace of the algebraic soliton $\nu(t)$ is contained in the straight line segment joining $\mu_0$ with the flat metric $\lambda$ given by $\lambda|_{\kg\times\ggo}=\mu_0|_{\kg\times\ggo}$, $\lambda|_{\pg\times\pg}=0$ (see e.g. the algebraic solitons $S^2\times\RR$, $H^2\times\RR$ and $Nil$ in \cite[Figure 1]{homRF}, and those denoted by $G_{bi}$, $E2$, $H\times\RR^m$ and $N$ in \cite[Figure 5]{homRF}).  Otherwise, if $\mu_{\kg},\mu_{\pg}\ne 0$, then $\nu(t)$ stays in the half parabola $\{ s^2\mu_{\kg}+s\mu_{\pg}:s>0\}$ joining $\mu_0$ with the flat $\lambda$  (e.g. the round metrics on $S^3$ in \cite[Figure 1]{homRF}).  In any case, the forward direction is determined by the sign of $c$.

\vs

We now study the evolution of semi-algebraic solitons under a normalized bracket flow.  Let $F:\hca_{q,n}\longrightarrow\RR$ be a function which is invariant under isometry (i.e. $F(\mu)=F(\lambda)$ for any pair $\mu,\lambda\in\hca_{q,n}$ of isometric homogenous spaces).  In particular, $F$ is $\left[\begin{smallmatrix} \Gl_q(\RR)&0\\ 0&\Or(n) \end{smallmatrix}\right]$-invariant relative to the action defined in (\ref{action})-(\ref{adkh}).  Also assume that $F$ is {\it scaling invariant}, in the sense that there exists $d\ne 0$ such that $F(c\cdot\mu)=c^dF(\mu)$  for any $c\in\RR$, $\mu\in\hca_{q,n}$.  Some examples of isometry and scaling invariant functions are given by the scalar curvature $R(\mu)$, more in general the other traced powers $\tr{\Ricci_\mu^k}$ of the Ricci operator, and $\|\nabla^k\Riem_\mu\|$, where $\nabla_\mu$ denotes the Levi-Civita connection and $\Riem_\mu$ the Riemann curvature tensor.

Consider the normalized bracket flow $\mu(t)$ as in (\ref{BFrn}) such that $F(\mu(t))\equiv F(\mu_0)$, and express it in terms of the unnormalized bracket flow by $\mu(t)=c(t)\nu(\tau(t))$ (see (\ref{ctau})).  If $\mu_0\in\hca_{q,n}$ is a Ricci soliton with cosmological constant $c_0$ (see \eqref{rseq}), then it follows from Theorem \ref{eqflrn} and (\ref{rssol}) that
$$
F(\mu_0)=c(t)^d(-2c_0\tau(t)+1)^{-d/2}F(\mu_0), \qquad\forall t.
$$
By assuming that $F(\mu_0)\ne 0$, we obtain that $c(t)=(-2c_0\tau(t)+1)^{1/2}$ and so $c'(t)=-c_0c(t)$, from which follows that $r(t)\equiv -c_0$ (recall from (\ref{ctau}) that $c'=rc$).  The evolution equation of the normalized bracket flow yielding $F$ constant and starting at a Ricci soliton $\mu_0$ is therefore given by
\begin{equation}\label{F-eq}
\ddt\mu=-\pi\left(\left[\begin{smallmatrix} 0&0\\ 0&\Ricci_{\mu}
\end{smallmatrix}\right]\right)\mu -c_0\mu, \qquad \mu(0)=\mu_0.
\end{equation}
Moreover, it follows that $c(t)=e^{-c_0t}$ and $\tau(t)=\tfrac{1-e^{-2c_0t}}{2c_0}$ (and $\tau(t)=t$ if $c_0=0$), which together with Proposition \ref{saBF} yield the following result.

\begin{proposition}\label{F-const}
Let $(G_{\mu_0}/K_{\mu_0},g_{\mu_0})$, $\mu_0 \in \hca_{q,n}$, be a homogeneous space that is a semi-algebraic soliton, say with
$$
\Ricci_{\mu_0}=cI+\unm(D_\pg+D_\pg^t), \qquad c\in\RR, \qquad D:=\left[\begin{smallmatrix} 0&\ast\\ 0&D_\pg \end{smallmatrix}\right]\in\Der(\ggo,\mu_0).
$$
Let $F:\hca_{q,n}\longrightarrow\RR$ be any differentiable function invariant under isometry and scaling such that $F(\mu_0)\ne 0$.  Then the normalized bracket flow solution such that $F(\mu(t))\equiv F(\mu_0)$ is given by
\begin{equation}\label{saevol_sc}
\mu(t) =  \left[\begin{smallmatrix} I&0\\ 0& e^{tA} e^{-tD_\pg} \end{smallmatrix}\right] \cdot \mu_0, \qquad t\in (-\infty,\infty),
\end{equation}
where $ A = \unm(D_\pg - D_\pg^t)$.
\end{proposition}

It follows from Remark \ref{saevol2} that if the derivation $D$ satisfies $D\pg \subseteq \pg$, and $F:\hca_{q,n}\longrightarrow\RR$ is only assumed to be scaling and $\left[\begin{smallmatrix} I&0\\ 0&\Or(n) \end{smallmatrix}\right]$-invariant, then the evolution is simply given by
\begin{equation}\label{F-const2}
\mu(t) =  \left[\begin{smallmatrix} I&0\\ 0& e^{tA}  \end{smallmatrix}\right] \cdot \mu_0.
\end{equation}
An example of a scaling and $\left[\begin{smallmatrix} I&0\\ 0&\Or(n) \end{smallmatrix}\right]$-invariant function which is not isometry invariant is $F(\mu)=\|\mu\|$ (see Section \ref{N-norm} below).

As the orbit $\left[\begin{smallmatrix} I&0\\ 0&\Or(n) \end{smallmatrix}\right]\cdot\mu_0$ is compact, the solution $\mu(t)$ in \eqref{F-const2} stays bounded and the limits of subsequences $\mu(t_k)$ are all isomorphic to $\mu_0$ (compare with Example \ref{evol-nosemi}).

On the other hand, recall that $A$ is skew-symmetric, so its eigenvalues are purely imaginary.  By Kronecker's theorem, there exists a sequence $t_k$, with $t_k \rightarrow \infty$, such that $e^{t_kA} \rightarrow I$.  This implies that $\mu(t_0+t_k) \underset{k \rightarrow \infty}\longrightarrow \mu(t_0)$ for any $t_0\in\RR$ and thus the whole solution is contained in its $\omega$-limit.  The absence of this kind of chaos for the bracket flow, which is still an open question, would imply that $\mu(t)\equiv\mu_0$, and so $A$ should be a derivation of $\mu_0$ thus obtaining that any semi-algebraic soliton would be algebraic.

\subsection{Normalizing by the bracket norm}\label{N-norm}
The choice of any inner product on $\kg$ allows us to consider the inner product on $\lamg$ defined by
\begin{equation}\label{innV}
\la\mu,\lambda\ra= \sum\la\mu(Y_i,Y_j),\lambda(Y_i,Y_j)\ra,
\end{equation}
where $\{ Y_i\}$ is the union of orthonormal bases of $\kg$ and $(\pg,\ip)$, respectively.  Given any $\mu\in\hca_{q,n}$, one may assume to make simpler some computations that the inner product on $\kg$ is $\Ad(K_\mu)$-invariant (recall that $\overline{\Ad(K_{\mu})}$ is compact in $\Gl(\ggo)$), which will also hold for any other element in $\hca_{q,n}$ coinciding with $\mu$ on $\kg\times\kg$ (e.g. for any normalized bracket flow solution starting at $\mu$).

The normalized bracket flow equation as in (\ref{BFrn}) such that the norm $\|\mu(t)\|$ of any solution remains constant in time, produces an interesting consequence: there always exists a convergent subsequence $\mu(t_k)\to\lambda$.  Recall that $\mu(t)|_{\kg\times\ggo}\equiv\mu_0|_{\kg\times\ggo}$ for any bracket flow solution, so we only need to keep $\|\mu(t)|_{\pg\times\pg}\|^2=\|\mu_{\kg}\|^2+\|\mu_{\pg}\|^2$ constant.  By using that
$$
\|c\cdot\mu|_{\pg\times\pg}\|^2=c^4\|\mu_{\kg}\|^2+c^2\|\mu_{\pg}\|^2,
$$
we obtain that $\|c\cdot\mu|_{\pg\times\pg}\|^2=1$ if and only if
\begin{equation}\label{cN-norm1}
c^2(t):= \frac{-\|\mu_{\pg}\|^2+\sqrt{\|\mu_{\pg}\|^4+4\|\mu_{\kg}\|^2}}{2\|\mu_{\kg}\|^2}, \qquad \mu_{\kg}\ne 0,
\end{equation}
and $c(t)=\|\mu_{\pg}\|^{-1}$ for $\mu_{\kg}=0$.  It is easy to prove that the corresponding normalizing function $r$ in equation (\ref{BFrn}) must be
$$
r=\frac{4\tr{\Ricci\mm}-\la\mu_{\kg},\mu_{\kg}(\Ricci\cdot,\cdot)+\mu_{\kg}(\cdot,\Ricci\cdot)\ra}{2\|\mu_{\kg}\|^2+\|\mu_{\pg}\|^2}.
$$
An alternative way to guaranty the existence of a convergent subsequence is by taking
\begin{equation}\label{cN-norm2}
c(t):=\frac{1}{\|\mu_{\kg}\|^{1/2}+\|\mu_{\pg}\|},
\end{equation}
as it follows that $0<\beta\leq\|c\cdot\mu|_{\pg\times\pg}\|^2\leq 2$, where $\beta=(1-\alpha)^2+\alpha^4\approx 0.289$ and $\alpha$ is the real root of $2x^3+x-1$.

Once we obtain a convergent subsequence $\mu(t_k)\to\lambda$, we have that $\lambda|_{\pg\times\pg}\ne 0$ as soon as $\mu_0|_{\pg\times\pg} \ne 0$, but we do not know if it is always the case that $\lambda\in\hca_{q,n}$.  The only condition in Definition \ref{hqn} which may fail is (h2), that is, $K_\lambda$ might not be closed in $G_\lambda$.  This would yield a collapsing with bounded geometry under the Ricci flow to the lower dimensional homogeneous space $G_\lambda/\overline{K_\lambda}$ (see \cite[Section 6.5]{spacehm}).  We also note that $\lambda$ may belong to $\hca_{q,n}$ and nevertheless be flat (see the examples in Sections \ref{evol-nosemi} and \ref{evol-Dpnop}).

Recall that $F(\mu)=\|\mu\|$ is scaling and $\left[\begin{smallmatrix} I&0\\ 0&\Or(n) \end{smallmatrix}\right]$-invariant, so that formula (\ref{F-const2}) applies for the evolution of semi-algebraic solitons.

\subsection{Scalar curvature normalization}\label{R-norm}
If we start the Ricci flow at a Ricci soliton $(M,g)$, the solution $g(t)$ is homothetic (i.e. isometric up to scaling) to $g$ for each $t$ where it is defined. So, in the homogeneous case, normalizing by constant in time scalar curvature yields a solution $g(t)$ that is isometric to $g$ for each $t$.

It is well known that the scalar curvature $R=R(g(t))$ of any Ricci flow solution $g(t)$ evolves by
$$
\dpar R=\Delta(R) +2\|\ricci\|^2,
$$
where $\Delta$ is the Laplace operator of the Riemannian manifold $(M,g(t))$ (see e.g. \cite[Lemma 6.7]{ChwKnp}).  As in the homogeneous case $R$ is constant on $M$, we simply get
\begin{equation}\label{evR}
\ddt R=2\|\ricci\|^2.
\end{equation}
We refer to \cite[Proposition 3.8, (vi)]{homRF} for an alternative proof of this fact within the homogeneous setting, as an application of Theorem \ref{eqflrn}.

\begin{lemma}\label{trRic2}
Let $(M,g)$ be a Ricci soliton with constant scalar curvature (not necessarily homogeneous), say with $\ricci(g) = c g + \lca_Xg$ and Ricci operator $\Ricci(g)$ . Then,
$$
cR(g)=\tr{\Ricci(g)^2},
$$
where $R(g)=\tr{\Ricci(g)}$ is the scalar curvature.
\end{lemma}

\begin{proof}
Let $g(t)$ be the unnormalized Ricci flow starting at $g$. It follows from \eqref{rssol} and \eqref{evR} that
\begin{align*}
2(-2ct+1)^{-2}\tr{\Ricci(g)^2} &= 2 \tr{\Ricci(g(t))^2} = \ddt R(g(t)) \\ &= \ddt (-2ct+1)^{-1} R(g) = 2c(-2ct+1)^{-2} R(g),
\end{align*}
and so the lemma follows.
\end{proof}

By using that a homogeneous manifold is flat if and only if it is Ricci flat (see \cite{AlkKml}), we deduce from the above lemma that a homogeneous Ricci soliton $(M,g)$ is flat as soon as $c=0$ or $R(g)=0$.  Furthermore, if $(M,g)$ is nonflat, then $c$ and $R(g)$ are both nonzero and have the same sign.

For the normalized bracket flow in this case, we have that the limit $\lambda$ of any subsequence $\mu(t_k)\to\lambda$, as $t_k\to\pm\infty$, is automatically nonflat as $R(\lambda)=R(\mu_0)$.  Unfortunately, the solution may diverge to infinity without any convergent subsequence.  There are examples of this behavior with $R<0$ in \cite[Sections 3.4 and 4]{homRF}, and to obtain examples with $R>0$ consider any product $E\times S$ of a compact Einstein homogeneous manifolds $E$ and a nonabelian flat solvable Lie group $S$.

Since $F(\mu)=R(\mu)$ is scaling and isometry invariant, we can apply Proposition \ref{F-const} and (\ref{F-const2}) to study the evolution of semi-algebraic solitons under this normalization.  In this case, these results also follow more directly by using \cite[Example 3.13]{homRF} and Lemma \ref{trRic2}.

\subsection{Example of a non semi-algebraic soliton evolution}\label{evol-nosemi}
Our aim in this section is to study the bracket flow evolution of the homogeneous Ricci soliton which is not semi-algebraic given in Example \ref{nosemi2}.  We therefore fix an orthonormal basis $\{X_1,Y_1,Z_1,X_2,Y_2,Z_2\}$ of $\ggo$ and consider $\nu=\nu_{a,b,c}\in\hca_{0,6}=\lca_6$ defined by
$$
\nu(X_1,Y_1)=aZ_1, \quad \nu(X_1,X_2)=bY_2, \quad \nu(X_1,Y_2)=-bX_2, \quad \nu(X_2,Y_2)=cZ_2.
$$
It is easy to see that for any $a,c\ne 0$, the solvmanifold $(G_\nu,\ip)$ is isometric to the nilsoliton $H_3\times H_3$, where $H_3$ denotes the $3$-dimensional Heisenberg group.  By a straightforward computation we obtain that the unnormalized bracket flow is equivalent to the ODE
$$
a'=-\tfrac{3}{2}a^3, \qquad b'=-\unm a^2b, \qquad c'=-\tfrac{3}{2}c^3,
$$
from which follows that if $a(0)=b(0)=c(0)=1$, then
$$
a=b^3, \qquad b(t)=(3t+1)^{-\tfrac{1}{6}}, \qquad c=b^3, \qquad t\in (-\tfrac{1}{3},\infty).
$$
Thus $\nu(t)\longrightarrow 0$ (flat), as $t\to\infty$, and $\nu(t)\longrightarrow \infty$, as $t\to-\tfrac{1}{3}$.

It follows from $\|\nu\|^2=2(a^2+2b^2+c^2)=4b^2(b^4+1)$ that
\begin{eqnarray*}
\frac{a}{\|\nu\|}&=&\frac{b^2}{2(b^4+1)^{1/2}}\underset{t\raw\infty}\longrightarrow 0, \qquad \bigg(\underset{{t\to-\tfrac{1}{3}}}\longrightarrow \unm\bigg), \\
\frac{b}{\|\nu\|}&=&\frac{1}{2(b^4+1)^{1/2}}\underset{t\raw\infty}\longrightarrow \unm, \qquad \bigg(\underset{{t\to-\tfrac{1}{3}}}\longrightarrow 0\bigg).
\end{eqnarray*}
This implies that under the bracket norm normalization, $\frac{\nu}{\|\nu\|}\longrightarrow\nu_{0,\unm,0}$, as $t\to\infty$, a nonabelian flat solvmanifold, and backward, as $t\to -\tfrac{1}{3}$,  we have that $\frac{\nu}{\|\nu\|}\longrightarrow\nu_{\unm,0,\unm}$, the nilsoliton $H_3\times H_3$ itself.

Concerning scalar curvature normalization, by using that $R=R(\nu)=-\unm(a^2+b^2)=-b^6$, we obtain
\[
\frac{a}{|R|^{1/2}}\equiv 1, \qquad
\frac{b}{|R|^{1/2}}=\frac{1}{b^2}\underset{t\to\infty}\longrightarrow \infty, \qquad \bigg(\underset{{t\to-\tfrac{1}{3}}}\longrightarrow 0 \bigg),
\]
and therefore $\frac{\nu}{|R|^{1/2}}\longrightarrow\infty$, as $t\to\infty$.  In the backward direction, as $t\to -\tfrac{1}{3}$, one has $\frac{\nu}{|R|^{1/2}}\longrightarrow\nu_{1,0,1}$, the nilsoliton  $H_3\times H_3$.

We note that all the limits obtained are non-isomorphic to the starting point $\nu_0=\nu_{1,1,1}$, and that $\nu(t)$ even diverges to infinity in one case.  This is in clear contrast with the evolution of semi-algebraic solitons described in (\ref{F-const2}).

\subsection{Example of an algebraic soliton evolution with $D\pg \nsubseteq \pg$}\label{evol-Dpnop}
By Proposition \ref{limrs}, (i), the fixed points of the normalized bracket flow are precisely algebraic solitons with $D\pg \subseteq \pg$. It is then natural to ask how an algebraic soliton with $D\pg \nsubseteq \pg$ evolves, and this is the question we address in this section, by studying the bracket flow evolution of the algebraic soliton given in Example \ref{Dpnop} in the specific case where $\ngo=\hg_3$ is the $3$-dimensional Heisenberg Lie algebra. Fix a basis $\{Z,X_1,X_2,X_3 \}$ of $\ggo$ and consider $\nu = \nu_{a,b,c}\in \hca_{1,3}$ defined by
\[
\left\{
\begin{array}{lll}
    \nu(Z, X_1) = X_2, & \nu(X_1,X_2) = a X_3 + b Z, & \nu(X_3,X_1) = cX_2,\\
    \nu(Z, X_2) = -X_1 &                             & \nu(X_2,X_3) = cX_1,
\end{array}
\right.
\]
where $\kg = \RR Z$ and $\{X_1,X_2,X_3\}$ is an orthonormal basis of $(\pg,\ip)$. For $(a,b,c)=(1,-1,1)$ we get the nilsoliton $H_3$ presented with the modified reductive decomposition, as in Example \ref{Dpnop}. The unnormalized bracket flow for $\nu_{a,b,c}$ is equivalent to the ODE
\[
a' = (-\tfrac32 a^2 + 2b + 2ac) a, \qquad b' = (-a^2 + 2b + 2ac)b, \qquad c' = \unm a^2 c.
\]
It is not difficult to see that if $b\neq 0$ then $\tfrac{ac}{b}$ remains constant, and so starting at $a(0) = c(0) = 1, b(0)=-1$ one easily solves the ODE and gets
\[
a = c^{-3}, \quad b = -c^{-2}, \qquad c(t) = (3t+1)^{\tfrac16}, \qquad t\in(-\tfrac13, \infty).
\]
We obtain $\nu(t) \lraw \infty$ if we let either $t\to \infty$ or $t\to -\tfrac13$.  This provides an explicit example of the following unexpected behavior: a bracket flow solution which is immortal but not due to uniform boundedness, as it goes to infinity.

Under the bracket norm normalization defined in \eqref{cN-norm2}, we see that
\[
\tfrac1{ \|\nu_{\kg}\|^{1/2}+\|\nu_{\pg}\| } \cdot \nu \underset{t\raw\infty}\lraw \nu_{0,0,\unm},
\]
a flat metric on the solvable Lie group $E(2)$, and backward,
\[
\tfrac1{ \|\nu_{\kg}\|^{1/2}+\|\nu_{\pg}\| } \cdot \nu \underset{t\raw -\tfrac13}\lraw \nu_{\tfrac1{\sqrt{2}},0,0},
\]
the nilsoliton $H_3$ itself, though presented with reductive decomposition such that $D\pg \subseteq \pg$. This follows from a straightforward calculation, by using that $\| \nu_\pg \|^2 = 2a^2 + 4c^2$, $\| \nu_\kg \| ^2 = 2 b^2$.

Regarding scalar curvature normalization, we have that $R = R(\nu) = -\unm a^2 = -\unm c^{-6}$, and then
\[
\frac{a}{|R|^{1/2}} \equiv \sqrt2, \qquad \frac{b}{|R|} = -2c^4, \qquad \frac{c}{|R|^{1/2}} = \sqrt{2} c^4.
\]
This implies that $\frac1{|R|^{1/2}} \cdot \nu \lraw \infty$ as $t\raw \infty$, and backward, one has that as $t\raw -\frac13$, $\frac1{|R|^{1/2}} \cdot \nu \lraw \nu_{\sqrt{2},0,0}$, the same nilsoliton $H_3$ obtained in the backward limit of the bracket norm normalization.

As in the previous example, we obtain non-isomorphic limits and even divergence in one case, in contrast with Proposition \ref{rsequiv} and \eqref{F-const2}, thus showing the advantages of having condition $D\pg \subseteq \pg$.

\section{A geometric characterization of algebraic solitons}\label{algdiag}

Whereas the concept of Ricci soliton is a Riemannian invariant, that is, invariant under isometry, the concept of semi-algebraic soliton is not, as it may depend on the presentation of the homogeneous manifold $(M,g)$ as a homogeneous space $(G/K,g)$ (see Section \ref{hrs}).  Moreover, being an algebraic soliton may a priori not only depend on such presentation, but also on the reductive decomposition $\ggo = \kg \oplus \pg$ one is choosing for the homogeneous space (see Definition \ref{as}).

The following property plays a key role in the study of the Ricci flow for homogeneous manifolds (see \cite{nicebasis}).

\begin{definition}\label{RFdiag}
A homogeneous manifold $(M,g)$ is said to be {\it Ricci flow diagonal} if at some point $p\in M$ there exists an orthonormal basis $\beta$ of $T_pM$ such that the Ricci flow solution $g(t)$ starting at $g$ is diagonal with respect to $\beta$ for any $t\in (T_-,T_+)$ (i.e. $g_{ij}(t)(p)=0$ for all $i\ne j$).
\end{definition}

We note that the point $p$ plays no role in this definition, as the condition holds either for every point or for none.  It is easy to see that the property of being Ricci flow diagonal is invariant under isometry.  Already in dimension $4$, there is a left-invariant metric on a nilpotent Lie group which is not Ricci flow diagonal (see \cite[Example 5.7]{nicebasis}).

Let $(M,g)$ be a homogeneous Ricci soliton, and consider any presentation $(G/K,g_{\ip})$ of $(M,g)$ with reductive decomposition $\ggo=\kg\oplus\pg$.  It is easy to check that if the unnormalized Ricci flow solution to equation \eqref{RFiprn} is written as
$$
\ip_t=\la P(t)\cdot,\cdot\ra,
$$
where $P(t)$ is the corresponding smooth curve of positive definite operators of $(\pg,\ip)$, then the Ricci flow equation is equivalent to the following
ODE for $P$:
\begin{equation}\label{RFP}
\ddt P=-2P\Ricci(\ip_t),
\end{equation}
where $\Ricci(\ip_t):=\Ricci(g(t))(o):\pg\longrightarrow\pg$ is the Ricci operator at the origin.  It follows from the uniqueness of ODE solutions that the following conditions are equivalent:

\begin{itemize}
\item $(M,g)$ is Ricci flow diagonal.

\item There exists an orthonormal basis $\beta$ of $(\pg,\ip)$ such that the matrix $[\Ricci(\ip_t)]_\beta$ is diagonal for all $t\in (T_-,T_+)$.

\item The family of symmetric operators $\{ P(t):t\in (T_-,T_+)\}$ is commutative.
\end{itemize}

By Remark \ref{saevol2} and Proposition \ref{rsequiv}, (iii), any algebraic soliton is Ricci flow diagonal.  We now prove that this condition actually characterizes algebraic solitons among homogeneous Ricci solitons.  In particular, if any, a homogeneous Ricci soliton which is not isometric to any algebraic soliton must be geometrically different from all known examples.

\begin{theorem}\label{diagalg}
A homogeneous Ricci soliton is Ricci flow diagonal if and only if it is isometric to an algebraic soliton.
\end{theorem}

\begin{proof}
Let $(M,g_0)=(G/K,g_0)$ be a homogeneous Ricci soliton presented as a semi-algebraic soliton, with Ricci operator $\Ricci(\ip_0) = cI + D_\pg - A$, $A = \unm (D_\pg - D_\pg^t)$, $\ip_0 = g_0(eK)$. Fix a reductive decomposition $\ggo = \kg \oplus \pg$ such that $D:=\left[\begin{smallmatrix} 0&0\\ 0&D_\pg \end{smallmatrix}\right] \in\Der(\ggo,\mu_0)$, and any inner product on $\ggo$ that extends $\ip_0$ and makes $\kg\perp \pg$. Therefore $D$ is normal if and only if $D_\pg$ is so.

Assume that $(G/K,g_0)$ is Ricci flow diagonal, hence
\[
[\Ricci(\ip_t), \Ricci(\ip_0)] = 0, \qquad \forall t\in(T_-,T_+),
\]
By using that $\Ricci(\ip_t) = (-2ct+1)^{-1}e^{s(t)D_\pg} \Ricci(\ip_0) e^{-s(t)D_\pg}$ (which follows from the proof of Proposition \ref{saBF}), we can rewrite the previous formula as
\[
[e^{sD_\pg}\Ricci(\ip_0)e^{-sD_{\pg}}, \Ricci(\ip_0)] = 0,  \qquad \forall s\in(-\epsilon,\epsilon).
\]
Now this implies that $[[D_\pg,\Ricci(\ip_0)],\Ricci(\ip_0)] = 0$, and so
$$
0=\tr{D_\pg[[D_\pg,\Ricci(\ip_0)],\Ricci(\ip_0)]}= -\tr{[D_\pg,\Ricci(\ip_0)]^2}.
$$
It follows that $[D_\pg,\Ricci(\ip_0)]=0$ as it is skew-symmetric, or equivalently, $[D_\pg,A]=0$, and thus $D_\pg$ is normal. Hence $D$ is normal as well, and so $D^t\in \Der(\ggo)$ since it is a well-known fact that the transpose of a normal derivation of a metric Lie algebra is again a derivation (see e.g. the proof of \cite[Lemma 4.7]{solvsolitons} or use that $\Aut(\ggo)$ is an algebraic group). Thus $\Ricci(\ip_0)-cI = \unm(D+D^t)_\pg$, with $\unm(D+D^t) \in \Der(\ggo)$, which shows that the semi-algebraic soliton is actually algebraic, concluding the proof.
\end{proof}


\begin{thebibliography}{MMMM}

\bibitem[AK75]{AlkKml} {\sc D. Alekseevskii, B. Kimel'fel'd},
Structure of homogeneous Riemannian spaces with zero Ricci curvature,
{\it Funktional Anal. i Prilozen} {\bf 9} (1975), 5-11
(English translation: {\it Functional Anal. Appl.} {\bf 9} (1975), 97-102.

\bibitem[Bs87]{Bss} {\sc A. Besse}, Einstein manifolds, {\it Ergeb. Math.} {\bf 10} (1987), Springer-Verlag,
Berlin-Heidelberg.

\bibitem[C09]{Cao} {\sc H-D Cao}, Recent progress on Ricci solitons, {\it Recent Advances in Geometric Analysis}, {\it Adv. Lect. Math.} {\bf 11},
                Higher Education Press/International Press, Beijing-Boston (2010), 1-38.

\bibitem[CK07]{ChwKnp}  {\sc B. Chow, D. Knopf}, The Ricci flow: An introduction, {\it AMS Math. Surv. Mon.}
{\bf 110} (2004), Amer. Math. Soc., Providence.

\bibitem[C$^+$07]{libro}  {\sc B. Chow, S.-C. Chu, D. Glickenstein, C.
Guenther, J. Isenberg, T, Ivey, D. Knopf, P. Lu, F. Luo, L. Ni}, The Ricci flow:
Techniques and Applications, Part I: Geometric Aspects, {\it AMS Math. Surv. Mon.}
{\bf 135} (2007), Amer. Math. Soc., Providence.

\bibitem[DHW12]{DncHllWng} {\sc A. Dancer, S. Hall, M. Wang}, Cohomogeneity One Shrinking Ricci Solitons: An Analytic and Numerical Study, {\it Asian J. Math.}, in press (arXiv).

\bibitem[J11]{Jbl}  {\sc M. Jablonski}, Homogeneous Ricci solitons, preprint 2011 (arXiv).

\bibitem[Lf12]{Lfn} {\sc R. Lafuente}, Solvsolitons associated with graphs, {\it Adv. Geom.}, in press.

\bibitem[LfL12]{alek} {\sc R. Lafuente, J. Lauret}, Structure of homogeneous Ricci solitons and the Alekseevskii's conjecture, in preparation.

\bibitem[L01]{soliton} {\sc J. Lauret}, Ricci soliton homogeneous nilmanifolds,
{\it Math. Ann.} \textbf{319} (2001), 715-733.

\bibitem[L09]{cruzchica}  \bysame, Einstein solvmanifolds and nilsolitons, {\it Contemp. Math.} {\bf 491} (2009), 1-35.

\bibitem[L11]{solvsolitons}  \bysame, Ricci soliton solvmanifolds, {\it J. reine angew. Math.} {\bf 650} (2011), 1-21.

\bibitem[L12a]{spacehm}  \bysame, Convergence of homogeneous manifolds, {\it J. London Math. Soc.}, in press (arXiv).

\bibitem[L12b]{homRF}  \bysame, Ricci flow of homogeneous manifolds, {\it Math. Z.}, in press (arXiv).

\bibitem[LW12]{nicebasis}  {\sc J. Lauret, C.E. Will}, On the diagonalization of the Ricci flow on Lie groups, {\it Proc. Amer. Math. Soc.}, in press (arXiv).

\bibitem[P10]{Pyn} {\sc T. Payne}, The Ricci flow for nilmanifolds, {\it J. Modern Dyn.} {\bf 4} (2010), 65-90.

\bibitem[W11]{Wll} {\sc C.E. Will}, The space of solvsolitons in low dimensions, {\it Ann. Global Anal. Geom.}, in press.

\bibitem[Wi11]{Wllm} {\sc M.B. Williams}, Explicit Ricci solitons on nilpotent Lie groups, {\it J. Geom. Anal.}, in press.

\bibitem[Wa12]{Wng} {\sc M.Y. Wang}, Einstein metrics from symmetry and bundle constructions: a sequel, {\it Differential Geometry: Under
                the  Influence of S.-S. Chern}, {\it Adv. Lect. Math.} {\bf 22},
                Higher Education Press/International Press, Beijing-Boston (2012), 253-309.
\end{thebibliography}
\end{document}